\documentclass{article}
\usepackage{amsfonts}
\usepackage{eurosym}
\usepackage{amsmath,amssymb,amsthm}

\newtheorem{theorem}{Theorem}[section]
\newtheorem{lemma}[theorem]{Lemma}
\newtheorem{corollary}[theorem]{Corollary}

\newtheorem{definition}[theorem]{Definition}

\newtheorem{remark}[theorem]{Remark}

\newcommand{\R}{\mathbb{R}}

\begin{document}

\title{Existence of solution for Hilfer fractional differential equations
with boundary value conditions \thanks{%
Mathematics Subject Classifications: 34A08, 26A33, 34A12,34A40.}}

\date{}
\author{Mohammed S. Abdo\thanks{%
Department of Mathematics, Dr.Babasaheb Ambedkar Marathwada University,
Aurangabad, (M.S) \textrm{431001}, India}\ , Satish K. Panchal\thanks{%
Department of Mathematics, Dr.Babasaheb Ambedkar Marathwada University,
Aurangabad, (M.S) \textrm{431001}, India}\ , Sandeep P. Bhairat \thanks{%
Faculty of Engineering Mathematics, Institute of Chemical Technology Mumbai, Marathwada Campus, Jalna (M.S),
India. Corresponding author email: sp.bhairat@marj.ictmumbai.edu.in}}
\maketitle

\begin{abstract}
In this paper, we consider a class of nonlinear fractional differential
equations involving Hilfer derivative with boundary conditions. First, we
obtain an equivalent integral for the given boundary value problem in
weighted space of continuous functions. Then we obtain the existence results
for a given problem under a new approach and minimal assumptions on
nonlinear function $f$. The technique used in the analysis relies on a
variety of tools including Schauder's, Schaefer's and Krasnosel'ski's fixed
point theorems. We demonstrate our results through illustrative examples.
\end{abstract}

\section{Introduction}

Fractional calculus (FC) is playing an even vital role in applied
mathematics and engineering sciences, provoking a blurring of boundaries
between scientific disciplines and the real world applications by a
resurgence of interest in the modern as well as classical techniques of
applied analysis, see \cite{rms,kd,rh,hi,kst}. The development of FC is a natural consequence of a high
level of excitement on the research frontier in applied analysis.

Fractional differential equations (FDEs) naturally occurs in many situations
and are studied intensively with initial and boundary value conditions over
the last three decades. The existence of a solution for such initial value
problems (IVPs) and boundary value problems (BVPs) is crucial for further
qualitative studies and applications. In recent years, an increasing
interest in the analysis of Hilfer FDEs has been developed in the
literature \cite{as,SP1,SPN,SP5,DB1,DBN,sp1,db,fk,hlt,rk,rc,sup,zrs,zt,dv,wz,zx}. We mention here some works on Hilfer fractional differential equations.

One of the first works in this direction with an initial value condition was the paper by K. M. Furati et al. \cite{fk}. They studied the Hilfer FDE
\begin{equation}\label{f1}
D_{a^{+}}^{\alpha,\beta }y(x)=f \left(x,y(x)\right), \qquad x>a,\,0<\alpha<1,\,0\leq\beta\leq1,
\end{equation}
with the initial condition
\begin{equation}\label{f2}
I_{a^{+}}^{1-\gamma }y(a^+)=y_a, \quad y_a\in\R, \,\, \gamma=\alpha+\beta(1-\alpha),
\end{equation}
where $D_{a^{+}}^{\alpha,\beta}$ is Hilfer fractional derivative of order $\alpha\in(0,1)$ and type $\beta\in[0,1],$ and $I_{a^{+}}^{1-\gamma}$ is Riemann-Liouville fractional integral of order $1-\gamma.$ The existence and uniqueness of solution to IVP \eqref{f1}-\eqref{f2} is proved in weighted space of continuous functions by using Banach fixed point theorem. For details, see \cite{fk,zt}.

In the year 2015, J. Wang and Y. Zang investigated the existence of a solution to
nonlocal IVP for Hilfer FDEs:
\begin{align}\label{l1}
D_{a^{+}}^{\alpha,\beta }u(t)&=f \left(t,u(t)\right), \qquad 0<\alpha<1, 0\leq\beta\leq1  ,t\in (a,b],\\
I_{a^{+}}^{1-\gamma }u(a^+)&=\sum_{i=1}^{m}\lambda_{i}u(\tau_{i}),\qquad \tau _{i}\in (a,b],\ \alpha\leq\gamma=\alpha+\beta-\alpha\beta,\label{l11}
\end{align}%
For details, see \cite{wz}.

Later, H. Gu and J. J. Trujillo \cite{ht} studied the existence of mild solution of Hilfer evolution equation:
\begin{align}\label{l2}
D_{0^{+}}^{\nu,\mu }x(t)&=Ax(t)+f\left(t,x(t)\right),\qquad 0<\alpha<1, 0\leq\beta\leq1  ,t\in (0,b],\\
I_{0^{+}}^{1-\gamma}x(0)&=x_0,\qquad x_0\in\R,\,\gamma=\alpha+\beta-\alpha\beta.\label{l22}
\end{align}%
They utilized the method of noncompact measure and established sufficient conditions to ensure the existence of a mild solution to Hilfer evolution IVP \eqref{l2}-\eqref{l22}. The state $x(t)$ defined for the values in Banach space $X$ with the norm $|\cdot|$ and $A$ is infinitesimal generator of $C_0$ semigroups in $X.$

In 2016, in \cite{rc}, Rafal Kamoki et al. considered fractional Cauchy problem involving Hilfer derivative
\begin{align}\label{l3}
D_{a^{+}}^{\alpha,\beta }y(t)&=g \left(t,y(t)\right), \qquad 0<\alpha<1,0\leq\beta\leq1,t\in [a,b],b>a,\\
I_{a^{+}}^{1-\gamma }y(a)&=c,\qquad c\in {\R}^n, \gamma=\alpha+\beta-\alpha\beta,\label{l33}
\end{align}%
and proved the existence and uniqueness of its solution in the space of continuous functions by using Banach contraction theorem. They used Bielecki norm without partitioning the interval and obtained solutions to both homogeneous and nonhomogeneous Cauchy problems.

In recent two years, the series of works on Hilfer FDEs have been published. S. Abbas et al. \cite{as} surveyed the existence and stability for Hilfer FDEs of the form:
\begin{align}\label{l4}
D_{0}^{\alpha,\beta}u(t)&=f\left(t,y(t),D_{0}^{\alpha,\beta}u(t)\right), \qquad 0<\alpha<1,0\leq\beta\leq1,t\in [0,\infty),\\
I_{0}^{1-\gamma}u(0)&=\phi,\qquad \phi\in {\R}, \gamma=\alpha+\beta-\alpha\beta,\label{l44}
\end{align}%
with the uniform norm on weighted space of bounded and continuous functions. They discussed existence, uniqueness and asymptotic stability of solution to IVP by using Schauder's fixed point theorem. Further, they obtained Ulam-type stabilities for Hilfer FDEs in Banach spaces using the measure of noncompactness and Monch's fixed point theorem. They also derived some results on the existence of weak solutions to \eqref{f1}-\eqref{f2}.

Z. Gao and X. Yu \cite{zx} discussed the existence of a solution to Hilfer integral BVP for the relaxation FDEs:
\begin{align}\label{l5}
D_{0^{+}}^{\nu,\mu }x(t)&=cx(t)+f\left(t,x(t)\right),\qquad c<0,0<\nu<1, 0\leq\mu\leq1,t\in (0,b],\\
I_{0^{+}}^{1-\gamma }x(0^+)&=\sum_{i=1}^{m}\lambda_{i}x(\tau_{i}),\qquad \tau _{i}\in (0,b),\ 0\leq\gamma=\nu+\mu-\nu\mu, \label{l55}
\end{align}
By utilizing properties of Mittag-Leffler function and fixed point theory, they established three existence results for the solution of Hilfer integral BVP \eqref{l5}-\eqref{l55} similar to that of results in \cite{hlt}.

Bhairat et al. in \cite{db} generalized IVP \eqref{f1}-\eqref{f2} for $\alpha\in(n-1,n).$
First, they derived equivalent integral representation in weighted space of continuous functions. Then by employing the method of successive approximations, the existence, uniqueness and continuous dependence of the solution are obtained. Further, in \cite{sp1}, Bhairat studied the singular IVP for Hilfer FDE:
\begin{align}\label{l7}
D_{a^+}^{\alpha,\beta}x(t)=f(t,x(t)),&\quad 0<\alpha<1,\, 0\leq\beta\leq1,\quad t>{a},\\
\displaystyle\lim_{t\to{a^{+}}}{(t-a)}^{1-\gamma}x(t)=x_0,&\qquad \gamma=\alpha+\beta(1-\alpha),
\end{align}
Using properties of Euler's beta, gamma functions and Picard's iterative technique, the existence and uniqueness of solution to the singular IVP were obtained. Some existence results for Hilfer-fractional implicit differential equation with nonlocal initial conditions can be found in \cite{as,dv}.

Recently, Suphawat et al \cite{sup} studied the nonlocal BVP:
\begin{equation}\label{l8}
D^{\alpha,\beta}x(t)=f(t,x(t)),\quad 1<\alpha<2,\, 0\leq\beta\leq1,\quad t\in[a,b],
\end{equation}
with the integral boundary conditions
\begin{equation}\label{l9}
x(a)=0,\quad x(b)=\sum_{i=1}^{m}\delta_{i}I^{\phi_i}x(\xi_i),\qquad \phi_{i}>0,\,\delta_i\in\R,\, \xi_i\in[a,b].
\end{equation}
The Banach contraction mapping principle, Banach fixed point theorem with Holder inequality, nonlinear contractions, Krasnoselskii's fixed point theorem, nonlinear Leray-Schauder alternative are employed to prove the existence of the solution to integral BVP.

Motivated by aforesaid works, in this paper, we consider the following BVP for a class of Hilfer FDEs:
\begin{equation}
D_{a^{+}}^{\alpha ,\beta }z(t)=f\big(t,z(t)\big),\text{ \ }0<\alpha
<1,\,0\leq \beta \leq 1,t\in (a,b],\qquad \ \qquad \qquad   \label{e8.1}
\end{equation}%
\begin{equation}
I_{a^{+}}^{1-\gamma }\big[cz(a^{+})+dz(b^{-})\big]=e_{k},\text{\ \ \ }\gamma
=\alpha +\beta (1-\alpha ),\, e_k\in\R, \label{e8.2}
\end{equation}%
where, $f:(a,b]\times \mathbb{R}\rightarrow \mathbb{R}$ be a function such that $%
f(t,z)\in C_{1-\gamma }[a,b]$ for any $z\in C_{1-\gamma }[a,b]$ and $%
c,d,e_{k}\in \mathbb{R}$. We obtain several existence results by Schauder's, Schaefer's and Krasnosel’ski's fixed point theorems.

The paper is organized as follows: Some preliminary concepts related to our problem are listed in Section 2 which are useful in the sequel. In Section 3, we first establish an equivalent integral equation of BVP \eqref{e8.1}-\eqref{e8.2} and then study the existence results. Illustrative examples are provided in the last section.

\section{Preliminaries}

In this section, we list some definitions, lemmas and weighted spaces
which are useful in the sequel.

Let $-\infty <a<b<+\infty .$ Let $C[a,b],AC[a,b]$ and $C^{n}[a,b]$ be the
spaces of continuous, absolutely continuous, $n-$times continuous and
continuously differentiable functions on $[a,b],$ respectively. Here $%
L^{p}(a,b),p\geq 1,$ is the space of Lebesgue integrable functions on $%
(a,b). $ Further more we recall the following weighted spaces \cite{fk}:
\begin{gather*}
C_{\gamma }[a,b]=\{g:(a,b]\rightarrow \mathbb{R}:(t-a)^{\gamma }g(t)\in
C[a,b]\},\quad 0\leq \gamma <1, \\
C_{\gamma }^{n}[a,b]=\{g:(a,b]\rightarrow \mathbb{R},g\in
C^{n-1}[a,b]:g^{(n)}(t)\in C_{\gamma }[a,b]\},\,n\in \mathbb{N}.
\end{gather*}

\begin{definition}
(\cite{kst}) Let $g:[a,\infty )\rightarrow R$ is a real valued continuous
function. The left sided Riemann-Liouville fractional integral of $g$ of
order $\alpha >0$ is defined by
\begin{equation}
I_{a^{+}}^{\alpha }g(t)=\frac{1}{\Gamma (\alpha )}\int_{a}^{t}(t-s)^{\alpha
-1}g(s)ds,\quad t>a,  \label{d1}
\end{equation}%
where $\Gamma (\cdot )$ is the Euler's Gamma function and $a\in
\mathbb{R}
.$ provided the right hand side is pointwise defined on $(a,\infty ).$
\end{definition}

\begin{definition}
(\cite{kst}) Let $g:[a,\infty )\rightarrow R$ is a real valued continuous
function. The left sided Riemann-Liouville fractional derivative of $g$ of
order $\alpha >0$ is defined by
\begin{equation}
D_{a^{+}}^{\alpha }g(t)=\frac{1}{\Gamma (n-\alpha )}\frac{d^{n}}{dt^{n}}%
\int_{a}^{t}(t-s)^{n-\alpha -1}g(s)ds,  \label{d2}
\end{equation}%
where $n=[\alpha ]+1,$ and $[\alpha ]$ denotes the integer part of $\alpha .$
\end{definition}

\begin{definition}
\label{7} (\cite{rh}) The left sided Hilfer fractional derivative of
function $g\in L^{1}(a,b)$ of order $0<\alpha <1$ and type $0\leq \beta
\leq 1$ is denoted as $D_{a^{+}}^{\alpha ,\beta }$ and defined by
\begin{equation}
D_{a^{+}}^{\alpha ,\beta }g(t)=I_{a^{+}}^{\beta (1-\alpha
)}DI_{a^{+}}^{(1-\beta )(1-\alpha )}g(t),\text{ }D=\frac{d}{%
dt}.  \label{d3}
\end{equation}%
where $I_{a^{+}}^{\alpha }$ and $D_{a^{+}}^{\alpha }$ are Riemann-Liouville
fractional integral and derivative defined by \eqref{d1} and \eqref{d2},
respectively.
\end{definition}

\begin{remark}
\label{rem8.a} From Definition \ref{7}, we observe that:

\begin{itemize}
\item[(i)] The operator $D_{a^{+}}^{\alpha ,\beta }$ can be written as
\begin{equation*}
D_{a^{+}}^{\alpha ,\beta }=I_{a^{+}}^{\beta (1-\alpha
)}DI_{a^{+}}^{(1-\gamma )}=I_{a^{+}}^{\beta (1-\alpha )}D^{\gamma
},~~~~~~~~\gamma =\alpha +\beta -\alpha \beta \text{.}
\end{equation*}

\item[(ii)] The Hilfer fractional derivative can be regarded as an
interpolator between the Riemann-Liouville derivative ($\beta =0$) and
Caputo derivative ($\beta =1$) as
\begin{equation*}
D_{a^{+}}^{\alpha ,\beta }=%
\begin{cases}
DI_{a^{+}}^{(1-\alpha )}=~D_{a^{+}}^{\alpha },~~~~~~~~~~if~\beta =0; \\
I_{a^{+}}^{(1-\alpha )}D=~^{c}D_{a^{+}}^{\alpha },~~~~~~~~if~\beta =1.%
\end{cases}%
\end{equation*}

\item[(iii)] In particular, if $0<\alpha <1,$ $0\leq \beta \leq 1$ and $%
\gamma =\alpha +\beta -\alpha \beta ,$ then%
\begin{equation*}
(D_{a^{+}}^{\alpha ,\beta }g)(t)=\Big(I_{a^{+}}^{\beta (1-\alpha )}\frac{d}{%
dt}\Big(I_{a^{+}}^{(1-\beta )(1-\alpha )}g\Big)\Big)(t).
\end{equation*}%
One has,
\begin{equation*}
(D_{a^{+}}^{\alpha ,\beta }g)(t)=\Big(I_{a^{+}}^{\beta (1-\alpha )}\Big(%
D_{a^{+}}^{\gamma }g\Big)\Big)(t),
\end{equation*}%
where $\Big(D_{a^{+}}^{\gamma }g\Big)(t)=\frac{d}{dt}\Big(%
I_{a^{+}}^{(1-\beta )(1-\alpha )}g\Big)(t).$
\end{itemize}
\end{remark}

\begin{definition}
(\cite{fk}) Let $0<\alpha <1,0\leq \beta \leq 1,$ the weighted space $%
C_{1-\gamma }^{\alpha ,\beta }[a,b]$ is defined by
\begin{equation}
C_{1-\gamma }^{\alpha ,\beta }[a,b]=\big\{g\in {C_{1-\gamma }[a,b]}%
:D_{a^{+}}^{\alpha ,\beta }g\in {C_{1-\gamma }[a,b]}\big\},\quad \gamma
=\alpha +\beta (1-\alpha ).  \label{w1}
\end{equation}%
Clearly, $D_{a^{+}}^{\alpha ,\beta }g=I_{a^{+}}^{\beta (1-\alpha
)}D_{a^{+}}^{\gamma }g$ and $C_{1-\gamma }^{\gamma }[a,b]\subset C_{1-\gamma
}^{\alpha ,\beta }[a,b],\,\gamma =\alpha +\beta -\alpha \beta $, $%
0<\alpha <1,0\leq \beta \leq 1.$ Consider the space $C_{\gamma }^{0}[a,b]$
with the norm
\begin{equation}
{\Vert g\Vert }_{C_{\gamma }^{n}}=\sum_{k=0}^{n-1}{\Vert g^{(k)}\Vert }_{C}+{%
\Vert g^{(n)}\Vert }_{C_{\gamma }}.  \label{n1}
\end{equation}
\end{definition}

\begin{lemma}
\label{def8.5} (\cite{kd}) If $\alpha >0$ and $\beta >0,$ and $g\in
L^{1}(a,b)$ for $t\in \lbrack a,b]$, then
\newline
$\Big(I_{a^{+}}^{\alpha }I_{a^{+}}^{\beta }g\Big)(t)=\Big(I_{a^{+}}^{\alpha
+\beta }g\Big)(t)$ and $\Big(D_{a^{+}}^{\alpha }I_{a^{+}}^{\beta }g\Big)%
(t)=g(t).$\newline
In particular, if $f\in C_{\gamma }[a,b]$ or $f\in C[a,b]$, then the above
properties hold for each $t\in (a,b]$ or $t\in \lbrack a,b]$ respectively.
\end{lemma}

\begin{lemma}
\label{Le1}(\cite{fk}) For $t>a,$ we have

\begin{description}
\item[(i)] $I_{a^{+}}^{\alpha }(t-a)^{\beta -1}=\frac{\Gamma (\beta )}{%
\Gamma (\beta +\alpha )}(t-a)^{\beta +\alpha -1},\quad \alpha \geq 0,\beta
>0.$\newline

\item[(ii)] $D_{a^{+}}^{\alpha }(t-a)^{\alpha -1}=0,\quad \alpha \in (0,1).$
\end{description}
\end{lemma}

\begin{lemma}
\label{def8.8} (\cite{fk}) Let $\alpha >0$, $\beta >0$ and $\gamma
=\alpha +\beta -\alpha \beta .$ If $g\in C_{1-\gamma }^{\gamma }[a,b]$, then%
\newline
\begin{equation*}
I_{a^{+}}^{\gamma }D_{a^{+}}^{\gamma }g=I_{a^{+}}^{\alpha }D_{a^{+}}^{\alpha
,\beta }g,~D_{a^{+}}^{\gamma }I_{a^{+}}^{\alpha }g=D_{a^{+}}^{\beta
(1-\alpha )}g.
\end{equation*}
\end{lemma}

\begin{lemma}
\label{Le2} (\cite{fk}) Let $0<\alpha <1,$ $0\leq \beta \leq 1$ and $g\in
C_{1-\gamma }[a,b].$ Then%
\begin{equation*}
I_{a^{+}}^{\alpha }D_{a^{+}}^{\alpha ,\beta }g(t)=g(t)-\frac{%
(t-a)^{\alpha +\beta (1-\alpha )-1}}{\Gamma (\alpha +\beta (1-\alpha ))}%
I_{a^{+}}^{(1-\beta )(1-\alpha )}g(a),\quad \text{for all}\quad t\in (a,b],
\end{equation*}%
Moreover, if $\ \gamma =\alpha +\beta -\alpha \beta ,$ $g\in C_{1-\gamma
}[a,b]$ and $I_{a^{+}}^{1-\gamma }g\in C_{1-\gamma }^{1}[a,b],$ then
\begin{equation*}
I_{a^{+}}^{\gamma }D_{a^{+}}^{\gamma }g(t)=g(t)-\frac{%
(t-a)^{\gamma -1}}{\Gamma (\gamma)}I_{a^{+}}^{1-\gamma }g(a),\quad
\text{for all}\quad t\in (a,b].
\end{equation*}
\end{lemma}

\begin{lemma}
\label{def8.7} (\cite{fk}) If $0\leq \gamma <1$ and $g\in C_{\gamma
}[a,b]$, then
\begin{equation*}
(I_{a^{+}}^{\alpha }g)(a)=\lim_{t\rightarrow a^{+}}I_{a^{+}}^{\alpha
}g(t)=0,~0\leq \gamma <\alpha .
\end{equation*}
\end{lemma}

\begin{lemma}
\label{le}(\cite{fk}) Let $\gamma =\alpha +\beta -\alpha \beta $ where $%
0<\alpha <1$ and $0\leq \beta \leq 1.$ Let $f:(a,b]\times \R\rightarrow \R$ be
a function such that $f(t,z)\in C_{1-\gamma }[a,b]$ for any $z\in
C_{1-\gamma }[a,b].$ If $z\in C_{1-\gamma }^{\gamma }[a,b],$ then $z$
satisfies IVP \eqref{f1}-\eqref{f2} if and only if $z$ satisfies the
Volterra integral equation
\begin{equation}
z(t)=\frac{y_a}{\Gamma (\gamma)}(t-a)^{\gamma -1}+\frac{%
1}{\Gamma (\alpha )}\int_{a}^{t}(t-s)^{\alpha -1}f(s,z(s))ds,\quad t>a.
\label{s3}
\end{equation}
\end{lemma}

\section{Existence of solution}

In this section we prove the existence of solution to BVP %
\eqref{e8.1}-\eqref{e8.2} in $C_{1-\gamma }^{\alpha ,\beta }[a,b].$

\begin{lemma}
\label{lee(1)} Let $0<\alpha <1$, $0\leq \beta \leq 1$ where $%
\gamma =\alpha +\beta -\alpha \beta $, and $f:(a,b]\times \mathbb{R}%
\rightarrow \mathbb{R}$ be a function such that $f(x,z)\in C_{1-\gamma
}[a,b] $ for any $z\in C_{1-\gamma }[a,b].$ If $z\in C_{1-\gamma }^{\gamma
}[a,b],$ then $z$ satisfies BVP \eqref{e8.1}-\eqref{e8.2} if and only if $z$
satisfies the integral equation$-$%
\begin{eqnarray}
z(t) &=&\frac{(t-a)^{\gamma -1}}{\Gamma (\gamma)}\frac{%
e_{k}}{d\left( 1+\frac{c}{d}\right) }-\frac{1}{\left( 1+\frac{c}{d}\right) }%
\frac{(t-a)^{\gamma -1}}{\Gamma (\gamma)}  \notag \\
&&\times \frac{1}{\Gamma (1-\gamma +\alpha )}\int_{a}^{b}(b-s)^{\alpha-\gamma}f(s,z(s))ds  \notag \\
&&+\frac{1}{\Gamma (\alpha )}\int_{a}^{t}(t-s)^{\alpha -1}f(s,z(s))ds.
\label{ee3}
\end{eqnarray}
\end{lemma}

Proof: \ In view of Lemma \ref{le}, the solution of \eqref{e8.1} can be
written as%
\begin{equation}
z(t)=\frac{I_{a^{+}}^{1-\gamma }z(a^{+})}{\Gamma (\gamma)}%
(t-a)^{\gamma -1}+\frac{1}{\Gamma (\alpha )}\int_{a}^{t}(t-s)^{\alpha
-1}f(s,z(s))ds,\quad t>a.  \label{e8.3}
\end{equation}%
\

Applying $I_{a^{+}}^{1-\gamma }$ on both sides of \eqref{e8.3} and taking
the limit $t\rightarrow b^{-}$, we obtain
\begin{equation}
I_{a^{+}}^{1-\gamma }z(b^{-})=I_{a^{+}}^{1-\gamma }z(a^{+})+%
\frac{1}{\Gamma (1-\gamma +\alpha )}\int_{a}^{b}(b-s)^{\alpha-\gamma}f(s,z(s))ds.  \label{e8.4}
\end{equation}%
Also again by applying $I_{a^{+}}^{1-\gamma }$ on both sides of \eqref{e8.3}%
, we have
\begin{eqnarray*}
I_{a^{+}}^{1-\gamma }z(t) &=&I_{a^{+}}^{1-\gamma }z(a^{+})+%
\frac{1}{\Gamma (1-\gamma +\alpha )}\int_{a}^{t}(t-s)^{\alpha-\gamma}f(s,z(s))ds \\
&=&I_{a^{+}}^{1-\gamma }z(a^{+})+I_{a^{+}}^{1-\beta (1-\alpha
)}f(t,z(t)).
\end{eqnarray*}%
Taking the limit $t\rightarrow a^{+}$\ and using Lemma \ref{def8.7} with $%
1-\gamma <1-\beta (1-\alpha ),$\ we obtain
\begin{equation}
I_{a^{+}}^{1-\gamma }z(a^{+})=I_{a^{+}}^{1-\gamma }z(a^{+}),
\label{e8.4a}
\end{equation}%
hence%
\begin{equation}
I_{a^{+}}^{1-\gamma }z(b^{-})=I_{a^{+}}^{1-\gamma }z(a^{+})+\frac{1}{\Gamma
(1-\gamma +\alpha )}\int_{a}^{b}(b-s)^{-\gamma +\alpha}f(s,z(s))ds.
\label{e8.4b}
\end{equation}

From the boundary condition (\ref{e8.2}), we have \
\begin{equation}
I_{a^{+}}^{1-\gamma }z(b^{-})=\frac{e_{k}}{d}-\frac{c}{d}I_{a^{+}}^{1-\gamma
}z(a^{+}).  \label{e3}
\end{equation}%
Comparing the equations (\ref{e8.4b}) and (\ref{e3}), and using (\ref{e8.4a}%
), we get%
\begin{equation}
I_{a^{+}}^{1-\gamma }z(a^{+})=\frac{1}{\left( 1+\frac{c}{d}\right) }\left(
\frac{e_{k}}{d}-\frac{1}{\Gamma (1-\gamma +\alpha )}\int_{a}^{b}(b-s)^{%
\alpha-\gamma}f(s,z(s))ds\right) .  \label{e4}
\end{equation}

Submitting \eqref{e8.3} into \eqref{e4}, we obtain%
\begin{eqnarray}
z(t) &=&\frac{(t-a)^{\gamma -1}}{\Gamma (\gamma)}\frac{%
e_{k}}{d\left( 1+\frac{c}{d}\right) }-\frac{1}{\left( 1+\frac{c}{d}\right) }%
\frac{(t-a)^{\gamma -1}}{\Gamma (\gamma)}  \notag \\
&&\times \frac{1}{\Gamma (1-\gamma +\alpha )}\int_{a}^{b}(b-s)^{\alpha-\gamma}f(s,z(s))ds  \notag \\
&&+\frac{1}{\Gamma (\alpha )}\int_{a}^{t}(t-s)^{\alpha -1}f(s,z(s))ds.
\label{E5}
\end{eqnarray}%
\

Conversely, applying $I_{a^{+}}^{1-\gamma }$ on both sides of \eqref{ee3},
using Lemmas \ref{Le1} and \ref{def8.5}, with some simple calculations, we
get
\begin{eqnarray*}
&&I_{a^{+}}^{1-\gamma }cz(a^{+})+I_{a^{+}}^{1-\gamma }dz(b^{-}) \\
&=&\frac{c}{\left( 1+\frac{c}{d}\right) }\left( \frac{e_{k}}{d}-\frac{1}{%
\Gamma (1-\gamma +\alpha )}\int_{a}^{b}(b-s)^{\alpha-\gamma}f(s,z(s))ds\right) \\
&&+\frac{d}{\left( 1+\frac{c}{d}\right) }\left( \frac{e_{k}}{d}-\frac{1}{%
\Gamma (1-\gamma +\alpha )}\int_{a}^{b}(b-s)^{\alpha-\gamma}f(s,z(s))ds\right) \\
&&+\frac{d}{\Gamma (1-\gamma +\alpha )}\int_{a}^{b}(b-s)^{\alpha-\gamma}f(s,z(s))ds \\
&=&\frac{ce_{k}}{\left( 1+\frac{c}{d}\right) d}+\frac{de_{k}}{\left( 1+\frac{%
c}{d}\right) d}-\left( \frac{c}{\left( 1+\frac{c}{d}\right) }+\frac{d}{%
\left( 1+\frac{c}{d}\right) }-d\right) \\
&&\frac{1}{\Gamma (1-\gamma +\alpha )}\int_{a}^{b}(b-s)^{\alpha-\gamma}f(s,z(s))ds \\
&=&e_{k}.
\end{eqnarray*}%
\ Which shows that the boundary condition (\ref{e8.2}) is satisfied. \newline

Next, applying $D_{a^{+}}^{\gamma }$ on both sides of \eqref{ee3} and using
Lemmas \ref{Le1} and \ref{def8.8}, we have
\begin{equation}
D_{a^{+}}^{\gamma }z(t)=D_{a^{+}}^{\beta (1-\alpha )}f\big(t,z(t)\big).
\label{e8.9}
\end{equation}

Since $z\in C_{1-\gamma }^{\gamma }[a,b]$ and by definition of $C_{1-\gamma
}^{\gamma }[a,b]$, we have $D_{a^{+}}^{\gamma }z\in C_{n-\gamma }[a,b]$,
therefore, $D_{a^{+}}^{\beta (1-\alpha )}f=DI_{a^{+}}^{1-\beta (1-\alpha
)}f\in C_{1-\gamma }[a,b].$ For $f\in C_{1-\gamma }[a,b]$, it is clear that $%
I_{a^{+}}^{1-\beta (1-\alpha )}f\in C_{1-\gamma }[a,b]$. Hence $f$ and $%
I_{a^{+}}^{1-\beta (1-\alpha )}f$ satisfy the hypothesis of Lemma \ref{Le2}.%
\newline
Now, applying $I_{a^{+}}^{\beta (1-\alpha )}$ on both sides of \eqref{e8.9},
we have%
\begin{equation*}
{\large I_{a^{+}}^{\beta (1-\alpha )}}D_{a^{+}}^{\gamma }z(t)={\large %
I_{a^{+}}^{\beta (1-\alpha )}}D_{a^{+}}^{\beta (1-\alpha )}f\big(t,z(t)\big).
\end{equation*}%
Using Remark~\ref{rem8.a} (i), Eq.(\ref{e8.9}) and Lemma \ref{Le2}, we get%
\begin{equation*}
I_{a^{+}}^{\gamma }D_{a^{+}}^{\gamma }z(t)=f\big(t,z(t)\big)-%
\frac{I_{a^{+}}^{1-\beta (1-\alpha )}f\big(a,z(a)\big)}{\Gamma (\beta
(1-\alpha ))}(t-a)^{\beta (1-\alpha )-1},\quad \text{for all}\quad t\in
(a,b].
\end{equation*}%
\ By Lemma \ref{def8.7}, we have $I_{a^{+}}^{1-\beta (n-\alpha )}f\big(a,z(a)%
\big)=0$. Therefore, we have $D_{a^{+}}^{\alpha ,\beta }z(t)=f\big(t,z(t)%
\big)$. This completes the proof.

Let us introduce the hypotheses needed to prove the existence of solutions
for the problem at hand.

\begin{itemize}
\item[ (H1)] $f:(a,b]\times
\mathbb{R}
\rightarrow
\mathbb{R}
$ is a function such that $f(\cdot ,z(\cdot ))\in C_{1-\gamma }^{\beta
(1-\alpha )}[a,b]$ for any $z\in C_{1-\gamma }[a,b]$ and there exist two
constants $N,\zeta >0$ such that%
\begin{equation*}
\left\vert f\big(t,z\big)\right\vert \leq N\big(1+\zeta \left\Vert
z\right\Vert _{C_{1-\gamma }}\big).
\end{equation*}

\item[ (H2)] The inequality
\begin{equation}
\mathcal{G}:=\frac{\Gamma (\gamma)}{\Gamma (\alpha +1)}\left[
(b-a)^{\alpha }+(b-a)^{\alpha +1-\gamma }\right] N\zeta <1  \label{rr1}
\end{equation}%
holds.
\end{itemize}

Now, we are ready to present the existence result for the BVP %
\eqref{e8.1}-\eqref{e8.2}, which is based on Schauder's fixed point theorem (see \cite{gd}).

\begin{theorem}
\label{th8.1} Assume that the hypotheses (H1) and (H2) are satisfied. Then
Hilfer boundary value problem \eqref{e8.1}-\eqref{e8.2} has at least one
solution in $C_{1-\gamma }^{\gamma }[a,b]\subset C_{1-\gamma }^{\alpha
,\beta }[a,b]$.
\end{theorem}

\begin{proof}
Define the operator ${\large \mathcal{T}}:C_{1-\gamma }[a,b]\longrightarrow
C_{1-\gamma }[a,b]$ by%
\begin{eqnarray}
\left( {\large \mathcal{T}}z\right) (t) &=&\frac{(t-a)^{\gamma-1}}{\Gamma (\gamma)}\frac{e_{k}}{d\left( 1+\frac{c}{d}\right) }-\frac{1%
}{\left( 1+\frac{c}{d}\right) }\frac{(t-a)^{\gamma -1}}{\Gamma
(\gamma)}  \notag \\
&&\times \frac{1}{\Gamma (1-\gamma +\alpha )}\int_{a}^{b}(b-s)^{-\gamma
+\alpha}f(s,z(s))ds  \notag \\
&&+\frac{1}{\Gamma (\alpha )}\int_{a}^{t}(t-s)^{\alpha -1}f(s,z(s))ds.
\label{e8.10}
\end{eqnarray}%
Let $\mathbb{B}_{r}=\left\{ z\in C_{1-\gamma }[a,b]:\left\Vert z\right\Vert
_{C_{1-\gamma }}\leq r\right\} $ with ${\large r\geq }\frac{\Omega }{1-%
\mathcal{G}},$ for $\mathcal{G<}1,$ where%
\begin{eqnarray*}
\Omega  &:&=\bigg[\frac{1}{\Gamma (\gamma)}%
\frac{e_{k}}{d\left( 1+\frac{c}{d}\right) }+\left\vert \frac{1}{1+\frac{c}{d}%
}\right\vert \frac{1}{\Gamma (\gamma)} \\
&&\times \left[ \frac{(b-a)^{\alpha +1-\gamma }}{\Gamma (2-\gamma +\alpha)%
}+\frac{(b-a)^{2\alpha +1-\gamma }}{\Gamma (\alpha +1)}\right] N\bigg].
\end{eqnarray*}
The proof will be given by the following three steps:

Step1: We show that $\mathcal{T}(\mathbb{B}_{r})\subset \mathbb{B}_{r}$. By
hypothesis $(H_{2})$, we have
\begin{align*}
& \left\vert (\mathcal{T}z)(t)(t-a)^{1-\gamma }\right\vert  \\
& \leq \left\vert \frac{1}{\Gamma (\gamma)}%
\frac{e_{k}}{d\left( 1+\frac{c}{d}\right) }\right\vert  \\
& +\left\vert \frac{1}{1+\frac{c}{d}}\frac{1}{\Gamma
(\gamma)}\right\vert \frac{1}{\Gamma (1-\gamma +\alpha )}%
\int_{a}^{b}(b-s)^{-\gamma +\alpha}N(1+\zeta \left\vert z\right\vert )ds
\\
& +\frac{\left\vert (t-a)^{1-\gamma }\right\vert }{\Gamma (\alpha )}%
\int_{a}^{t}(t-s)^{\alpha -1}N(1+\zeta \left\vert z\right\vert )ds \\
& \leq \frac{1}{\Gamma (\gamma)}\frac{e_{k}}{%
d\left( 1+\frac{c}{d}\right) }+\left\vert \frac{1}{1+\frac{c}{d}}\right\vert
\frac{1}{\Gamma (\gamma)} \\
& \times \frac{1}{\Gamma (1-\gamma +\alpha )}\int_{a}^{b}(b-s)^{\alpha-\gamma}N(1+\zeta (s-a)^{\gamma -1}\Vert z\Vert _{C_{1-\gamma }})ds \\
& +\frac{\left\vert (t-a)^{1-\gamma }\right\vert }{\Gamma (\alpha )}%
\int_{a}^{t}(t-s)^{\alpha -1}N(1+\zeta (s-a)^{\gamma -1}\Vert z\Vert
_{C_{1-\gamma }})ds.
\end{align*}%
Note that, for any $z\in \mathbb{B}_{r}$, and for each $t\in (a,b]$, we get
\begin{eqnarray*}
&&\frac{1}{\Gamma (1-\gamma +\alpha )}\int_{a}^{b}(b-s)^{-\gamma +\alpha}N(1+\zeta (s-a)^{\gamma -1}{\large \Vert z\Vert _{C_{1-\gamma }}})ds \\
&\leq &\left[ \frac{(b-a)^{1-\gamma }}{\Gamma (2-\gamma +\alpha )}+\frac{%
\zeta r\Gamma (\gamma)}{\Gamma (\alpha +1)}\right] N(b-a)^{\alpha },
\end{eqnarray*}%
and%
\begin{eqnarray*}
&&\frac{\left\vert (t-a)^{1-\gamma }\right\vert }{\Gamma (\alpha )}%
\int_{a}^{t}(t-s)^{\alpha -1}N(1+\zeta (s-a)^{\gamma -1}{\large \Vert z\Vert
_{C_{1-\gamma }}})ds \\
&\leq &\left[ \frac{(t-a)^{\alpha }}{\Gamma (\alpha +1)}+\frac{\zeta r\Gamma
(\gamma)}{\Gamma (\alpha +1)}\right] N(t-a)^{\alpha +1-\gamma }.
\end{eqnarray*}%
Hence%
\begin{eqnarray*}
\left\vert (\mathcal{T}z)(t)(t-a)^{1-\gamma }\right\vert  &\leq
&\frac{1}{\Gamma (\gamma)}\frac{e_{k}}{d\left(
1+\frac{c}{d}\right) }+\left\vert \frac{1}{1+\frac{c}{d}}\right\vert
\frac{1}{\Gamma (\gamma)} \\
&&\times \left[ \frac{(b-a)^{1-\gamma }}{\Gamma (2-\gamma +\alpha)}+\frac{%
\zeta r\Gamma (\gamma)}{\Gamma (\alpha +1)}\right] N(b-a)^{\alpha } \\
&&+\left[ \frac{(t-a)^{\alpha }}{\Gamma (\alpha +1)}+\frac{\zeta r\Gamma
(\gamma)}{\Gamma (\alpha +1)}\right] N(t-a)^{\alpha +1-\gamma },
\end{eqnarray*}%
which yields%
\begin{eqnarray*}
{\large \Vert \mathcal{T}z\Vert _{C_{1-\gamma }}} &\leq &\frac{%
1}{\Gamma (\gamma)}\frac{e_{k}}{d\left( 1+\frac{c}{d}\right) }%
+\left\vert \frac{1}{1+\frac{c}{d}}\right\vert\frac{%
1}{\Gamma (\gamma)} \\
&&\times \bigg[\frac{(b-a)^{\alpha +1-\gamma }}{\Gamma (2-\gamma +\alpha)}%
+\frac{(b-a)^{2\alpha +1-\gamma }}{\Gamma (\alpha +1)}\bigg]N \\
&&+\frac{N\zeta r\Gamma (\gamma)}{\Gamma (\alpha +1)}\bigg[%
(b-a)^{\alpha }+(b-a)^{\alpha +1-\gamma }\bigg]
\end{eqnarray*}%
In the light of hypothesis (H2) and definition of $r$, we get $\Vert {\large
\mathcal{T}}z\Vert _{C_{1-\gamma }}\leq \mathcal{G}r+(1-\mathcal{G})r=r,$
that is, ${\large \mathcal{T(}}\mathbb{B}_{r})\subset \mathbb{B}_{r}.$

We shall prove that ${\large \mathcal{T}}$ is completely continuous.

Step 2. The operator ${\large \mathcal{T}}$ is continuous. Suppose that $%
\{z_{n}\}$ is a sequence such that $z_{n}\rightarrow z$ in $\mathbb{B}_{r}$\
as $n\rightarrow \infty $. Then for each $t\in (a,b],$ we have%
\begin{eqnarray*}
&&\left\vert \big((\mathcal{T}z_{n})(t)-(\mathcal{T}z)(t)\big)%
(t-a)^{1-\gamma }\right\vert  \\
&=&\left\vert \frac{1}{1+\frac{c}{d}}\right\vert\frac{%
1}{\Gamma (\gamma)} \\
&&\times \frac{1}{\Gamma (1-\gamma +\alpha )}\int_{a}^{b}(b-s)^{-\gamma
+\alpha}\left\vert f\big(s,z_{n}(s)\big)-f\big(s,z(s)\big)\right\vert ds
\\
&&+\frac{(t-a)^{1-\gamma }}{\Gamma (\alpha )}\int_{a}^{t}(t-s)^{\alpha
-1}\left\vert f\big(s,z_{n}(s)\big)-f\big(s,z(s)\big)\right\vert ds \\
&\leq &\left\vert \frac{1}{1+\frac{c}{d}}\right\vert\frac{%
1}{\Gamma (\gamma)} \\
&&\times \frac{\Gamma (\gamma)}{\Gamma (\alpha +1)}(b-a)^{\alpha
}\left\Vert f\big(\cdot ,z_{n}(\cdot )\big)-f\big(\cdot ,z(\cdot )\big)%
\right\Vert _{C_{1-\gamma }} \\
&&+\frac{\Gamma (\gamma)(t-a)^{1-\gamma +\alpha }}{\Gamma (-\gamma+\alpha)}\left\Vert f\big(\cdot ,z_{n}(\cdot )\big)-f\big(\cdot
,z(\cdot )\big)\right\Vert _{C_{1-\gamma }}.
\end{eqnarray*}%
\ Since $f$ is continuous on $(a,b]$, and $z_{n}\rightarrow z,$\ this
implies
\begin{equation*}
\Vert (\mathcal{T}z_{n}-\mathcal{T}z)\Vert _{C_{1-\gamma }}\longrightarrow
0~~as~~n\longrightarrow \infty ,
\end{equation*}%
which means that operator $\mathcal{T}$ is continuous on $\mathbb{B}_{r}$.%
\newline

Step 3. We prove that $\mathcal{T}(\mathbb{B}_{r})$ is relatively compact.
From Step 1, we have ${\large \mathcal{T(}}\mathbb{B}_{r})\subset \mathbb{B}%
_{r}.$ It follows that ${\large \mathcal{T(}}\mathbb{B}_{r})$ is uniformly
bounded. Moreover, we show that operator $\mathcal{T}$ is equicontinuous on $%
\mathbb{B}_{r}$. Indeed,for any $a<t_{1}<t_{2}<b$ and $z\in \mathbb{B}_{r}$,
we get%
\begin{eqnarray*}
&&\left\vert (t_{2}-a)^{1-\gamma }\big({\large \mathcal{T}}z\big)%
(t_{2})-(t_{1}-a)^{1-\gamma }\big({\large \mathcal{T}}z\big)%
(t_{1})\right\vert  \\
&\leq &\frac{\left\vert
(t_{2}-a)^{^{n-k}}-(t_{1}-a)^{^{n-k}}\right\vert }{\Gamma (\gamma)}%
\frac{e_{k}}{d\left( 1+\frac{c}{d}\right) }+\left\vert \frac{1}{1+\frac{c}{d}%
}\right\vert \frac{\left\vert
(t_{2}-a)^{^{n-k}}-(t_{1}-a)^{^{n-k}}\right\vert }{\Gamma (\gamma)} \\
&&\times \frac{1}{\Gamma (1-\gamma +\alpha )}\int_{a}^{b}(b-s)^{1-\gamma
+\alpha -1}\left\vert f\big(s,z(s)\big)\right\vert ds \\
&&+\dfrac{1}{\Gamma (\alpha )}\left\vert (t_{2}-a)^{1-\gamma
}\int_{a}^{t_{2}}(t_{2}-s)^{\alpha -1}f\big(s,z(s)\big)ds\right.  \\
&&\left. -(t_{1}-a)^{1-\gamma }\int_{a}^{t_{1}}(t_{1}-s)^{\alpha -1}f\big(%
s,z(s)\big)ds\right\vert  \\
&\leq &\frac{\left\vert
(t_{2}-a)^{^{n-k}}-(t_{1}-a)^{^{n-k}}\right\vert }{\Gamma (\gamma)}%
\left[ \frac{e_{k}}{d\left( 1+\frac{c}{d}\right) }\right.  \\
&&\left. +\left\vert \frac{1}{1+\frac{c}{d}}\right\vert \frac{\left\Vert
f\right\Vert _{C_{1-\gamma }}}{\Gamma (1-\gamma +\alpha )}%
\int_{a}^{b}(b-s)^{-\gamma +\alpha}(s-a)^{\gamma -1}ds\right]  \\
&&+\dfrac{\Vert f\Vert _{C_{1-\gamma }}}{\Gamma (\alpha )}\left\vert
(t_{2}-a)^{1-\gamma }\int_{a}^{t_{2}}(t_{2}-s)^{\alpha -1}(s-a)^{\gamma
-1}ds\right.  \\
&&\left. -(t_{1}-a)^{1-\gamma }\int_{a}^{t_{1}}(t_{1}-s)^{\alpha
-1}(s-a)^{\gamma -1}ds\right\vert  \\
&\leq &\frac{\left\vert
(t_{2}-a)^{^{n-k}}-(t_{1}-a)^{^{n-k}}\right\vert }{\Gamma (\gamma)}%
\left[ \frac{e_{k}}{d\left( 1+\frac{c}{d}\right) }\right.  \\
&&+\left. \left\vert \frac{1}{1+\frac{c}{d}}\right\vert \frac{\Gamma (\gamma)}{\Gamma (\alpha +1)}(b-a)^{\alpha }\left\Vert f\right\Vert
_{C_{1-\gamma }}\right]  \\
&&+\dfrac{\Vert f\Vert _{C_{1-\gamma }}}{\Gamma (\alpha )}\mathcal{B}(\gamma
-n+1,\alpha )\left\vert (t_{2}-a)^{\alpha }-(t_{1}-a)^{\alpha }\right\vert
\end{eqnarray*}%
 which tends to zero as $t_{2}\rightarrow t_{1},$ independent of $z\in
\mathbb{B}_{r}$, where $\mathcal{B(\cdot },\mathcal{\cdot )}$ is a Beta
function. Thus we conclude that $\mathcal{T}(\mathbb{B}_{r})$ is
equicontinuous on $\mathbb{B}_{r}$ and therefore $\mathcal{T}(\mathbb{B}_{r})
$ is relatively compact.  As a consequence of Steps 1 to 3 together with
Arzela-Ascoli theorem, we can conclude that $\mathcal{T}:\mathbb{B}%
_{r}\rightarrow \mathbb{B}_{r}$ is completely continuous operator.

An application of Schauder's fixed point theorem shows that there exists at
least a fixed point $z$ of $\mathcal{T}$ in $C_{1-\gamma }[a,b]$. This fixed
point $z$ is the solution to (\ref{e8.1})-(\ref{e8.2}) in $C_{1-\gamma
}^{\gamma }[a,b],$ and the proof is completed.
\end{proof}

We will study the next existence result by using Schaefer fixed point
theorem. To this end, we change hypothesis (H1) into the following one:

\begin{itemize}
\item[ (H3)] $f:(a,b]\times
\mathbb{R}
\rightarrow
\mathbb{R}
$ is a function such that $f(\cdot ,z(\cdot ))\in C_{1-\gamma }^{\beta
(1-\alpha )}[a,b]$ for any $z\in C_{1-\gamma }[a,b]$ and there exist a
function $\eta (t)\in C_{1-\gamma }[a,b]$ such that%
\begin{equation*}
\left\vert f\big(t,z\big)\right\vert \leq \eta (t),\text{ for all }t\in
(a,b],\text{ }z\in
\mathbb{R}
.
\end{equation*}
\end{itemize}

\begin{theorem}
\label{th3.3} Assume that $\ $(H3) holds. Then Hilfer boundary value problem %
\eqref{e8.1}-\eqref{e8.2} has at least one solution in $C_{1-\gamma
}^{\gamma }[a,b]\subset C_{1-\gamma }^{\alpha ,\beta }[a,b]$.
\end{theorem}

\begin{proof}
As in the proof of Theorem \ref{th8.1}, one can repeat Steps 1 to 3 and show
that operator ${\large \mathcal{T}}:C_{1-\gamma }[a,b]\longrightarrow
C_{1-\gamma }[a,b]$ defined in (\ref{e8.10}) is a completely continuous. It
remains to prove that%
\begin{equation*}
\Delta =\left\{ z\in {\large C_{n-\gamma }[a,b]}:z=\lambda \mathcal{T}z,%
\text{ for some }\lambda \in (0,1)\right\}
\end{equation*}%
is a bounded set. Let $z\in \Delta $ and $\lambda \in (0,1)$ be such that $%
z=\lambda Tz.$ By hypothesis (H3) and Eq.(\ref{e8.10}), then for all $t\in
\lbrack a,b],$ we have%
\begin{align*}
& \left\vert {\large \mathcal{T}z(t)(t-a)}^{n-\gamma }\right\vert  \\
& \leq \sum_{k=1}^{n}\frac{(t-a)^{n-k}}{\Gamma (\gamma -k+1)}\frac{e_{k}}{%
d\left( 1+\frac{c}{d}\right) } \\
& +\left\vert \frac{1}{1+\frac{c}{d}}\right\vert \sum_{k=1}^{n}\frac{%
(t-a)^{n-k}}{\Gamma (\gamma -k+1)}\frac{1}{\Gamma (n-\gamma +\alpha )}%
\int_{a}^{b}(b-s)^{n-\gamma +\alpha -1}\eta (s)ds \\
& +\frac{\left\vert (t-a)^{n-\gamma }\right\vert }{\Gamma (\alpha )}%
\int_{a}^{t}(t-s)^{\alpha -1}\eta (s)ds \\
& \leq \sum_{k=1}^{n}\frac{(t-a)^{n-k}}{\Gamma (\gamma -k+1)}\frac{e_{k}}{%
d\left( 1+\frac{c}{d}\right) } \\
& +\left\vert \frac{1}{1+\frac{c}{d}}\right\vert \sum_{k=1}^{n}\frac{%
(t-a)^{n-k}}{\Gamma (\gamma -k+1)}\frac{1}{\Gamma (n-\gamma +\alpha )}%
\int_{a}^{b}(b-s)^{n-\gamma +\alpha -1}(s-a)^{\gamma -n}\Vert \mathcal{\eta }%
\Vert _{C_{n-\gamma }}ds \\
& +\frac{\left\vert (t-a)^{n-\gamma }\right\vert }{\Gamma (\alpha )}%
\int_{a}^{t}(t-s)^{\alpha -1}(s-a)^{\gamma -n}\Vert \mathcal{\eta }\Vert
_{C_{n-\gamma }}ds.
\end{align*}%
That is%
\begin{eqnarray}
&&\Vert \mathcal{T}z\Vert _{C_{n-\gamma }}  \notag \\
&\leq &\sum_{k=1}^{n}\frac{(b-a)^{n-k}}{\Gamma (\gamma -k+1)}\frac{e_{k}}{%
d\left( 1+\frac{c}{d}\right) }  \notag \\
&&+\left[ \left\vert \frac{1}{1+\frac{c}{d}}\right\vert \sum_{k=1}^{n}\frac{%
(b-a)^{-k}}{\Gamma (\gamma -k+1)}\frac{\Gamma (\alpha )}{\mathcal{B(\alpha }%
,1\mathcal{)}}+\frac{\mathcal{B(}\gamma -n+1,1\mathcal{)}}{\Gamma (\alpha
)(b-a)^{\gamma }}\right] (b-a)^{n+\alpha }\Vert \mathcal{\eta }\Vert
_{C_{n-\gamma }}  \notag \\
&:&=\ell .  \label{ss3}
\end{eqnarray}

Since $\lambda \in (0,1)$ then $z<\mathcal{T}z.$ The last inequality with
Eq.(\ref{ss3}) leads us to%
\begin{equation*}
\Vert z\Vert _{C_{n-\gamma }}<\Vert \mathcal{T}z\Vert _{C_{n-\gamma }}\leq
\ell
\end{equation*}%
Which proves that $\Delta $ is bounded. By using Schaefer fixed point
theorem, the proof can be completed.
\end{proof}

Finally, we present the existence result for the problem \eqref{e8.1}-%
\eqref{e8.2}, which is based on Krasnosel'skii fixed point theorem. For this
end, we change hypothesis (H1) into the following one:

\begin{itemize}
\item[ (H4)] $f:(a,b]\times
\mathbb{R}
\rightarrow
\mathbb{R}
$ is a function such that $f(\cdot ,z(\cdot ))\in C_{n-\gamma }^{\beta
(n-\alpha )}[a,b]$ for any $z\in C_{n-\gamma }[a,b]$ and there exists
constant $L>0$ such that%
\begin{equation*}
\left\vert f\big(t,z\big)-f\big(t,w\big)\right\vert \leq L\left\vert
z-w\right\vert ,\text{ }\forall t\in (a,b],\text{ }z,w\in
\mathbb{R}%
\end{equation*}
And we consider the following hypothesis:

\item[ (H5)] The inequality%
\begin{eqnarray*}
\mathcal{W} &:&=\bigg[\left\vert \frac{1}{1+\frac{c}{d}}\right\vert
\sum_{k=1}^{n}\frac{(b-a)^{n-k}}{\Gamma (\gamma -k+1)}+\frac{\mathcal{B}%
(\gamma -n,\alpha +1)}{\Gamma (\gamma -n)}\bigg] \\
&&\times \frac{\Gamma (\gamma -n)(b-a)^{\alpha }}{\mathcal{B}(\gamma
-n,1)\Gamma (\alpha +1)}L<1
\end{eqnarray*}
is holds.
\end{itemize}

\begin{theorem}
\label{th3.4} Assume that hypotheses (H4) and (H5) are satisfied. If%
\begin{equation}
\left\vert \frac{1}{1+\frac{c}{d}}\right\vert \sum_{k=1}^{n}\frac{%
(b-a)^{n-k+\alpha }}{\Gamma (\gamma -k+1)}\frac{\Gamma (\gamma -n+1)}{\Gamma
(\alpha +1)}L<1.  \label{e2}
\end{equation}
Then Hilfer boundary value problem \eqref{e8.1}-\eqref{e8.2} has at least
one solution in $C_{n-\gamma }^{\gamma }[a,b]\subset C_{n-\gamma }^{\alpha
,\beta }[a,b]$.
\end{theorem}

\begin{proof}
Consider the operator ${\large \mathcal{T}}$ is defined as in Theorem \ref%
{th8.1}.

Now, we need to analyze the operator ${\large \mathcal{T}}$ into sum two
operators ${\large \mathcal{T}}_{1}+{\large \mathcal{T}}_{2}$ as follows%
\begin{equation*}
{\large \mathcal{T}}_{1}z(t)=-\frac{1}{\left( 1+\frac{c}{d}\right) }%
\sum_{k=1}^{n}\frac{(t-a)^{\gamma -k}}{\Gamma (\gamma -k+1)}\frac{1}{\Gamma
(n-\gamma +\alpha )}\int_{a}^{b}(b-s)^{n-\gamma +\alpha -1}f(s,z(s))ds
\end{equation*}%
and
\begin{equation*}
{\large \mathcal{T}}_{2}z(t)=\sum_{k=1}^{n}\frac{(t-a)^{\gamma -k}}{\Gamma
(\gamma -k+1)}\frac{e_{k}}{d\left( 1+\frac{c}{d}\right) }+\frac{1}{\Gamma
(\alpha )}\int_{a}^{t}(t-s)^{\alpha -1}f(s,z(s))ds.
\end{equation*}%
\

Set $\widetilde{f}=f(s,0)$ and consider the ball $\mathbb{B}_{\epsilon
}=\{z\in C_{n-\gamma ;\psi }([a,b]:\left\Vert z\right\Vert _{C_{n-\gamma
;\psi }}\leq \epsilon \}\ $with $\epsilon \geq \frac{\Lambda }{1-\mathcal{W}}%
,\mathcal{W<}1$ where
\begin{eqnarray}
\Lambda  &=&\bigg[\left\vert \frac{1}{1+\frac{c}{d}}\right\vert
\sum_{k=1}^{n}\frac{(b-a)^{n-k}}{\Gamma (\gamma -k+1)}+\frac{\mathcal{B}%
(\gamma -n,\alpha +1)}{\Gamma (\gamma -n)}\bigg]  \notag \\
&&\times \frac{\Gamma (\gamma -n)(b-a)^{\alpha }}{\mathcal{B}(\gamma
-n,1)\Gamma (\alpha +1)}\left\Vert \widetilde{f}\right\Vert _{C_{n-\gamma
}}+\sum_{k=1}^{n}\frac{(b-a)^{n-k}}{\Gamma (\gamma -k+1)}\frac{e_{k}}{%
d\left( 1+\frac{c}{d}\right) }.  \label{E1}
\end{eqnarray}

The proof will be given in three stages.

\textbf{Stage 1:} We prove that ${\large \mathcal{T}}_{1}z+{\large \mathcal{T%
}}_{2}w\in \mathbb{B}_{r}$ for every $z,w\in \mathbb{B}_{\epsilon }.$

By assumpition (H4), then for every $z\in \mathbb{B}_{\epsilon },$and $t\in
(a,b]$, we have%
\begin{eqnarray*}
&&\left\vert (t-a)^{n-\gamma }{\large \mathcal{T}}_{1}z(t)\right\vert  \\
&\leq &\left\vert \frac{1}{1+\frac{c}{d}}\right\vert \sum_{k=1}^{n}\frac{%
(t-a)^{n-k}}{\Gamma (\gamma -k+1)}\frac{1}{\Gamma (n-\gamma +\alpha )} \\
&&\times \int_{a}^{b}(b-s)^{n-\gamma +\alpha -1}\bigg[\left\vert
f(s,z(s))-f(s,0)\right\vert +\left\vert f(s,0)\right\vert \bigg]ds \\
&\leq &\left\vert \frac{1}{1+\frac{c}{d}}\right\vert \sum_{k=1}^{n}\frac{%
(t-a)^{n-k}}{\Gamma (\gamma -k+1)}\frac{1}{\Gamma (n-\gamma +\alpha )} \\
&&\times \int_{a}^{b}(b-s)^{n-\gamma +\alpha -1}(s-a)^{\gamma -n}\bigg[%
L\left\Vert z\right\Vert _{C_{n-\gamma }}+\left\Vert \widetilde{f}%
\right\Vert _{C_{n-\gamma }}\bigg]ds \\
&\leq &\left\vert \frac{1}{1+\frac{c}{d}}\right\vert \sum_{k=1}^{n}\frac{%
(t-a)^{n-k}}{\Gamma (\gamma -k+1)}\frac{\Gamma (\gamma -n+1)}{\Gamma (\alpha
+1)}(t-a)^{\alpha }\bigg[L\epsilon +\left\Vert \widetilde{f}\right\Vert
_{C_{n-\gamma }}\bigg].
\end{eqnarray*}

This gives%
\begin{eqnarray}
&&\left\Vert {\large \mathcal{T}}_{1}z\right\Vert _{C_{n-\gamma }}  \notag \\
&\leq &\left\vert \frac{1}{1+\frac{c}{d}}\right\vert \sum_{k=1}^{n}\frac{%
(b-a)^{n-k+\alpha }}{\Gamma (\gamma -k+1)}\frac{\Gamma (\gamma -n+1)}{\Gamma
(\alpha +1)}\bigg[L\epsilon +\left\Vert \widetilde{f}\right\Vert
_{C_{n-\gamma }}\bigg].  \label{q1}
\end{eqnarray}

For operator ${\large \mathcal{T}}_{2}$, we have
\begin{eqnarray*}
&&\left\vert (t-a)^{n-\gamma }{\large \mathcal{T}}_{2}w(t)\right\vert  \\
&\leq &\sum_{k=1}^{n}\frac{(t-a)^{n-k}}{\Gamma (\gamma -k+1)}\frac{e_{k}}{%
d\left( 1+\frac{c}{d}\right) } \\
&&+\frac{(t-a)^{n-\gamma }}{\Gamma (\alpha )}\int_{a}^{t}(t-s)^{\alpha -1}%
\bigg[\left\vert f(s,w(s))-f(s,0)\right\vert +\left\vert f(s,0)\right\vert %
\bigg]ds \\
&\leq &\sum_{k=1}^{n}\frac{(t-a)^{n-k}}{\Gamma (\gamma -k+1)}\frac{e_{k}}{%
d\left( 1+\frac{c}{d}\right) } \\
&&+\frac{(t-a)^{n-\gamma }}{\Gamma (\alpha )}\int_{a}^{t}(t-s)^{\alpha
-1}(s-a)^{\gamma -n}\bigg[L\left\Vert w\right\Vert _{C_{n-\gamma
}}+\left\Vert \widetilde{f}\right\Vert _{C_{n-\gamma }}\bigg]ds.
\end{eqnarray*}

For every $w\in \mathbb{B}_{\epsilon },$and $t\in (a,b]$, this gives%
\begin{eqnarray}
\left\Vert {\large \mathcal{T}}_{2}w\right\Vert _{C_{n-\gamma }} &\leq
&\sum_{k=1}^{n}\frac{(b-a)^{n-k}}{\Gamma (\gamma -k+1)}\frac{e_{k}}{d\left(
1+\frac{c}{d}\right) }  \notag \\
&&+\frac{\Gamma (\gamma -n+1)(b-a)^{\alpha }}{\Gamma (\gamma -n+\alpha +1)}%
\bigg[L\epsilon +\left\Vert \widetilde{f}\right\Vert _{C_{n-\gamma }}\bigg].
\label{q2}
\end{eqnarray}

From Eqs.(\ref{q1}),(\ref{q2}), and using hypothesis (H5) with Eq.(\ref{E1}%
), we get%
\begin{eqnarray*}
&&\left\Vert {\large \mathcal{T}}_{1}z+{\large \mathcal{T}}_{2}w\right\Vert
_{C_{n-\gamma }} \\
&\leq &\left\Vert {\large \mathcal{T}}_{1}z\right\Vert _{C_{n-\gamma
}}+\left\Vert {\large \mathcal{T}}_{2}w\right\Vert _{C_{n-\gamma }} \\
&\leq &\bigg[\left\vert \frac{1}{1+\frac{c}{d}}\right\vert \sum_{k=1}^{n}%
\frac{(b-a)^{n-k}}{\Gamma (\gamma -k+1)}+\frac{\Gamma (\alpha +1)}{\Gamma
(\gamma -n+\alpha +1)}\bigg] \\
&&\times \frac{\Gamma (\gamma -n+1)(b-a)^{\alpha }}{\Gamma (\alpha +1)}\bigg[%
L\epsilon +\left\Vert \widetilde{f}\right\Vert _{C_{n-\gamma }}\bigg] \\
&&+\sum_{k=1}^{n}\frac{(b-a)^{n-k}}{\Gamma (\gamma -k+1)}\frac{e_{k}}{%
d\left( 1+\frac{c}{d}\right) } \\
&\leq &\mathcal{W}\epsilon +(1-\mathcal{W})\epsilon =\epsilon .
\end{eqnarray*}

This proves that ${\large \mathcal{T}}_{1}z+{\large \mathcal{T}}_{2}w\in
\mathbb{B}_{r}$ for every $z,w\in \mathbb{B}_{r}.$

\textbf{Stage 2:} We prove that the operator ${\large \mathcal{T}}_{1}$ is a
contration mapping on $\mathbb{B}_{r}.$

For any $z,w\in \mathbb{B}_{r},$ and for $t\in (a,b],$ then by assumptions
(H4), we have%
\begin{eqnarray*}
&&\left\vert (t-a)^{n-\gamma }{\large \mathcal{T}}_{1}z(t)-(t-a)^{n-\gamma }%
{\large \mathcal{T}}_{1}w(t)\right\vert  \\
&\leq &\left\vert \frac{1}{1+\frac{c}{d}}\right\vert \sum_{k=1}^{n}\frac{%
(t-a)^{n-k}}{\Gamma (\gamma -k+1)}\frac{1}{\Gamma (n-\gamma +\alpha )} \\
&&\times \int_{a}^{b}(b-s)^{n-\gamma +\alpha -1}\bigg[\left\vert
f(s,z(s))-f(s,w(s))\right\vert \bigg]ds \\
&\leq &\left\vert \frac{1}{1+\frac{c}{d}}\right\vert \sum_{k=1}^{n}\frac{%
(t-a)^{n-k}}{\Gamma (\gamma -k+1)} \\
&&\times \frac{1}{\Gamma (n-\gamma +\alpha )}\int_{a}^{b}(b-s)^{n-\gamma
+\alpha -1}L\left\vert z(s)-w(s))\right\vert ds \\
&\leq &\left\vert \frac{1}{1+\frac{c}{d}}\right\vert \sum_{k=1}^{n}\frac{%
(t-a)^{n-k}}{\Gamma (\gamma -k+1)}\frac{\Gamma (\gamma -n+1)}{\Gamma (\alpha
+1)}(b-s)^{\alpha }L\left\Vert z-w\right\Vert _{C_{n-\gamma }}
\end{eqnarray*}%
This gives,%
\begin{eqnarray*}
&&\left\Vert {\large \mathcal{T}}_{1}z-{\large \mathcal{T}}_{1}w\right\Vert
_{C_{n-\gamma }} \\
&\leq &\left\vert \frac{1}{1+\frac{c}{d}}\right\vert \sum_{k=1}^{n}\frac{%
(b-a)^{n-k+\alpha }}{\Gamma (\gamma -k+1)}\frac{\Gamma (\gamma -n+1)}{\Gamma
(\alpha +1)}L\left\Vert z-w\right\Vert _{C_{n-\gamma }}.
\end{eqnarray*}

The operator ${\large \mathcal{T}}_{1}$ is contraction mapping due to
assumption Eqs.(\ref{e2}).

\textbf{Stage 3:} We show that the operator ${\large \mathcal{T}}_{2}$ is
completely continuous on $\mathbb{B}_{\epsilon }.$

Firstly, from the continuity of $f$, we conclude that the operator ${\large
\mathcal{T}}_{2}:\mathbb{B}_{\epsilon }\rightarrow \mathbb{B}_{\epsilon }$
i.e. ${\large \mathcal{T}}_{2}$ is continuous on $\mathbb{B}_{\epsilon }$.
Next, we show that for all $\epsilon >0,$ there exists some $\epsilon
^{^{\prime }}>0$ such that $\left\Vert {\large \mathcal{T}}_{2}z\right\Vert
_{C_{n-\gamma }}\leq \epsilon ^{^{\prime }}.$ According to stage 1, for $%
z\in \mathbb{B}_{\epsilon },$ we know that%
\begin{eqnarray*}
\left\Vert {\large \mathcal{T}}_{2}z\right\Vert _{C_{n-\gamma }} &\leq
&\sum_{k=1}^{n}\frac{(b-a)^{n-k}}{\Gamma (\gamma -k+1)}\frac{e_{k}}{d\left(
1+\frac{c}{d}\right) } \\
&&+\mathcal{B}(\gamma -n+1,\alpha )\frac{(b-a)^{\alpha }}{\Gamma (\alpha )}%
\bigg[L\epsilon +\left\Vert \widetilde{f}\right\Vert _{C_{n-\gamma }}\bigg].
\end{eqnarray*}%
which is independent of $t$ and $z$, hence there exists
\begin{equation*}
\epsilon ^{^{\prime }}=\sum_{k=1}^{n}\frac{(b-a)^{n-k}}{\Gamma (\gamma -k+1)}%
\frac{e_{k}}{d\left( 1+\frac{c}{d}\right) }+\mathcal{B}(\gamma -n+1,\alpha )%
\frac{(b-a)^{\alpha }}{\Gamma (\alpha )}\bigg[L\epsilon +\left\Vert
\widetilde{f}\right\Vert _{C_{n-\gamma }}
\end{equation*}%
such that $\left\Vert {\large \mathcal{T}}_{2}z\right\Vert _{C_{n-\gamma
}}\leq \epsilon ^{^{\prime }}.$ So ${\large \mathcal{T}}_{2}$ is uniformly
bounded set on $\mathbb{B}_{\epsilon }.$

Finally, to prove that ${\large \mathcal{T}}_{2}$ is equicontinuous in $%
\mathbb{B}_{\epsilon }$, for any $z\in \mathbb{B}_{\epsilon }$ and $%
t_{1},t_{2}\in (a,b]$ with $t_{1}<t_{2},$ we have%
\begin{eqnarray*}
&&\left\vert (t_{2}-a)^{n-\gamma }\mathcal{T}_{2}z(t_{2})-(t_{1}-a)^{n-%
\gamma }\mathcal{T}_{2}z(t_{1})\right\vert  \\
&=&\left\vert \sum_{k=1}^{n}\frac{(t_{2}-a)^{n-k}-(t_{1}-a)^{n-k}}{\Gamma
(\gamma -k+1)}\frac{e_{k}}{d\left( 1+\frac{c}{d}\right) }\right.  \\
&&+\frac{(t_{2}-a)^{n-\gamma }}{\Gamma (\alpha )}%
\int_{a}^{t_{2}}(t_{2}-s))^{\alpha -1}f(s,z(s))ds \\
&&-\left. \frac{(t_{1}-a)^{n-\gamma }}{\Gamma (\alpha )}%
\int_{a}^{t_{1}}(t_{1}-s))^{\alpha -1}f(s,z(s))ds\right\vert  \\
&\leq &\sum_{k=1}^{n}\frac{\left\vert
(t_{2}-a)^{n-k}-(t_{1}-a)^{n-k}\right\vert }{\Gamma (\gamma -k+1)}\frac{e_{k}%
}{d\left( 1+\frac{c}{d}\right) } \\
&&\left\vert +\frac{(t_{2}-a)^{n-\gamma }}{\Gamma (\alpha )}%
\int_{a}^{t_{2}}(t_{2}-s))^{\alpha -1}(s-a)^{\gamma -n}\left\Vert
f\right\Vert _{C_{n-\gamma ;\psi }[a,b]}ds\right.  \\
&&\left. -\frac{(t_{1}-a)^{n-\gamma }}{\Gamma (\alpha )}%
\int_{a}^{t_{1}}(t_{1}-s))^{\alpha -1}(s-a)^{\gamma -n}\left\Vert
f\right\Vert _{C_{n-\gamma ;\psi }[a,b]}ds\right\vert  \\
&=&\sum_{k=1}^{n}\frac{\left\vert (t_{2}-a)^{n-k}-(t_{1}-a)^{n-k}\right\vert
}{\Gamma (\gamma -k+1)}\frac{e_{k}}{d\left( 1+\frac{c}{d}\right) } \\
&&+\frac{\mathcal{B}(\gamma -n+1)}{\Gamma (\alpha )}\left\Vert f\right\Vert
_{C_{n-\gamma ;\psi }[a,b]}\left\vert (t_{2}-a)^{\alpha }-(t_{1}-a)^{\alpha
}\right\vert .
\end{eqnarray*}%
Observe that the right-hand side of the above inequality is independent of $%
z.$ So
\begin{equation*}
\left\vert (t_{2}-a)^{n-\gamma }{\large \mathcal{T}}%
_{2}z(t_{2})-(t_{1}-a)^{n-\gamma }{\large \mathcal{T}}_{2}z(t_{1})\right\vert \rightarrow 0,\text{ as }|t_{2}-t_{1}|\rightarrow 0.
\end{equation*}%
This proves that ${\large \mathcal{T}}_{2}$ is equicontinuous on $\mathbb{B}%
_{\epsilon }$. In the view of Arzela-Ascoli Theorem, it follows that $(%
{\large \mathcal{T}}_{2}\mathbb{B}_{\epsilon })$ is relatively compact. As a
consequence of Krasnosel'skii fixed point theorem, we conclude that the
problem \eqref{e8.1}-\eqref{e8.2} has at least one solution.
\end{proof}

\begin{corollary}
Assume that hypotheses (H4) and (H5) are satisfied. Then Hilfer boundary
value problem \eqref{e8.1}-\eqref{e8.2} has a unique solution in $%
C_{n-\gamma }^{\gamma }[a,b]\subset C_{n-\gamma }^{\alpha ,\beta }[a,b]$.
\end{corollary}

\section{An example \label{Sec5}}

Consider the Hilfer fractional differential equation with boundary condition%
\begin{equation}
\begin{cases}
D_{a^{+}}^{\alpha ,\beta }z(t)=f\big(t,z(t)\big),t\in (0,1],0<\alpha
<1,0\leq \beta \leq 1, \\
I_{a^{+}}^{1-\gamma }\big[\frac{1}{4}z(0^{+})+\frac{3}{4}z(1^{-})\big]=\frac{%
2}{5},~~\alpha \leq \gamma =\alpha +\beta -\alpha \beta ,%
\end{cases}
\label{3}
\end{equation}%
where, $\alpha =\frac{1}{2},\beta =\frac{1}{3}$, $\gamma =\frac{2}{3}$, $c=%
\frac{1}{4}$, $d=\frac{3}{4}$, $e_{1}=\frac{2}{5},$ and%
\begin{equation*}
f\big(t,z(t)\big)=t^{\frac{-1}{6}}+\frac{1}{16}t^{\frac{5}{6}}\sin z(t),
\end{equation*}%
Clearly, $t^{\frac{1}{3}}f\big(t,z(t)\big)=t^{\frac{1}{6}}+\frac{1}{16}t^{%
\frac{7}{6}}\sin z(t)\in C[0,1],$ hence $f\big(t,z(t)\big)\in C_{\frac{1}{3}%
}[0,1].$ Observe that, for any $z\in
\mathbb{R}
^{+}$ and $t\in (0,1],$

\begin{eqnarray*}
\left\vert f\big(t,z(t)\big)\right\vert &\leq &t^{\frac{1}{6}}\left( 1+\frac{%
1}{16}t^{\frac{2}{3}}\left\vert t^{\frac{1}{3}}z(t)\right\vert \right) \\
&\leq &\left( 1+\frac{1}{16}\left\Vert z\right\Vert _{C_{\frac{1}{3}%
}}\right) .
\end{eqnarray*}

Therefore, the conditions (H1) is satisfied with $N=1,$ and $\zeta =\frac{1}{%
16}$. It is easy to check that the (H2)  is satisfied too. Indeed,$\ $by
some calculations, we get%
\begin{equation}
\mathcal{G}=\frac{\Gamma (\gamma -n+1)}{\Gamma (\alpha +1)}\left[
(b-a)^{\alpha }+(b-a)^{\alpha +n-\gamma }\right] N\zeta \simeq 0.19<1.
\notag
\end{equation}

An application of Theorem \ref{th8.1} implies that problem (\ref{3}) has a
solution in $C_{\frac{1}{3}}^{\frac{2}{3}}([0,1])$.

Moreover, consider $f\big(t,z(t)\big)=t^{\frac{-1}{6}}+\frac{1}{16}t^{\frac{5%
}{6}}\sin (t),$ it follows $\left\vert f\big(t,z(t)\big)\right\vert \leq t^{%
\frac{-1}{6}}+\frac{1}{16}t^{\frac{5}{6}}=\eta (t)\in C_{1-\gamma }[0,1].$
Therefore (H3) holds. An application of Theorem \ref{th3.3} implies that
problem (\ref{3}) has a solution in $C_{\frac{1}{3}}^{\frac{2}{3}}([0,1])$.

Finally, if $f\big(t,z(t)\big)=t^{\frac{-1}{6}}+\frac{1}{16}t^{\frac{5}{6}%
}\sin z(t),$ then for $z,w\in
\mathbb{R}
^{+}$ and $t\in (0,1],$ we get%
\begin{equation*}
\left\vert f\big(t,z(t)\big)-f\big(t,w(t)\big)\right\vert \leq \frac{1}{16}%
\left\vert z-w\right\vert .
\end{equation*}%
Thus, the hypothesis $(H4)$ is satisfied with $L=\frac{1}{16}$. It is easy
to check that hypothesis (H5) \ and inequality (\ref{e2}) are satisfied.
Indeed,$\ $by some calculations, we get%
\begin{eqnarray*}
\mathcal{W} &:&=\bigg[\left\vert \frac{1}{1+\frac{c}{d}}\right\vert
\sum_{k=1}^{n}\frac{(b-a)^{n-k}}{\Gamma (\gamma -k+1)}+\frac{\mathcal{B}%
(\gamma -n,\alpha +1)}{\Gamma (\gamma -n)}\bigg] \\
&&\times \frac{\Gamma (\gamma -n)(b-a)^{\alpha }}{\mathcal{B}(\gamma
-n,1)\Gamma (\alpha +1)}L\simeq 0.14<1
\end{eqnarray*}%
and%
\begin{equation}
\left\vert \frac{1}{1+\frac{c}{d}}\right\vert \sum_{k=1}^{n}\frac{%
(b-a)^{n-k+\alpha }}{\Gamma (\gamma -k+1)}\frac{\Gamma (\gamma -n+1)}{\Gamma
(\alpha +1)}L\simeq 0.05<1.
\end{equation}%
An application of Theorem \ref{th3.4} implies that problem (\ref{3}) has a
solution in $C_{\frac{1}{3}}^{\frac{2}{3}}[0,1]$.

\section{Conclusion}

We have obtained some existence results for the solution of boundary value
problem for Hilfer fractional differential equations based on the reduction
of fractional differential equations to integral equations. The employed
techniques, the fixed point theorems, are quite general and effective. We
trust the reported results here will have a positive impact on the
development of further applications in engineering and applied sciences.

\end{document}